\documentclass[3p, numbers]{elsarticle}
\usepackage{amsthm}
\usepackage{mathtools}
\usepackage{bbm}
\usepackage{tikz-cd}
\usepackage{hyperref}

\usepackage{amssymb}
\usepackage{amsmath}
\usepackage{amsfonts}
\usepackage{amsthm}
\usepackage{float}
\usepackage{times}
\usepackage{color}
\usepackage{proof}
\newtheorem{theorem}{Theorem}[section]
\newtheorem{proposition}[theorem]{Proposition}
\newtheorem{lemma}[theorem]{Lemma}

\usepackage{mdframed}
\usepackage{mathtools}

\usepackage{float}
\usepackage{tikz}
\usetikzlibrary{positioning,arrows}
\usetikzlibrary{decorations.pathreplacing}
\tikzset{
modal/.style={>=stealth,shorten >=1pt,shorten <=1pt,auto,node distance=1.5cm,semithick},
world/.style={circle,draw,minimum size=0.5cm,fill=gray!15},
point/.style={circle,draw,inner sep=0.5mm,fill=black},
reflexive above/.style={->,loop,looseness=7,in=120,out=60},
reflexive below/.style={->,loop,looseness=7,in=240,out=300},
reflexive left/.style={->,loop,looseness=7,in=150,out=210},
reflexive right/.style={->,loop,looseness=7,in=30,out=330}}

\newenvironment{newlist}
   {\begin{list}{}{\setlength{\labelsep}{0.15cm}
\setlength{\itemsep}{0cm}
\setlength{\topsep}{0.15cm}
                   \setlength{\labelwidth}{0.7cm}
                      \setlength{\leftmargin}{0.75cm}}}  
   {\end{list}}
   
\newcommand\nbd[1]{\protect\nobreakdash#1\hspace{0pt}}

\begin{document}
\begin{frontmatter}
\journal{ArXiv}
\title{A Study of Subminimal Logics of Negation and their Modal Companions}
\author[dj]{Nick~Bezhanishvili}
\ead{N.Bezhanishvili@uva.nl}
\author[ac]{Almudena~Colacito}
\address[ac]{Mathematical Institute, University of Bern \\ Sidlerstrasse 5, 3012 Bern, Switzerland}
\ead{almudena.colacito@math.unibe.ch}
\author[dj]{Dick~de~Jongh}
\address[dj]{Institute for Logic, Language and Computation, \\ University of Amsterdam, The Netherlands}
\ead{D.H.J.deJongh@uva.nl}


\begin{abstract}We study propositional logical systems arising from the language of Johansson's minimal  logic and obtained by weakening the requirements for the negation operator. We present their semantics as a variant of neighbourhood semantics. We use duality and completeness results to show that there are uncountably many subminimal logics. We also give model-theoretic and algebraic definitions of filtration for minimal logic  and show that they are dual to each other. These constructions ensure that the propositional minimal logic has the finite model property.  Finally, we define and investigate bi-modal companions with non-normal modal operators for some relevant subminimal systems, and give infinite axiomatizations for these bi-modal companions. 
\end{abstract}


\end{frontmatter}

\section{Introduction}

Minimal propositional calculus (\emph{Minimalkalk\"ul}, denoted here as $\mathsf{MPC}$) is the system obtained from the positive fragment of intuitionistic propositional calculus (equivalently, positive logic~\cite{rasiowa1974algebraic}) by adding a unary negation operator satisfying the so-called principle of contradiction (sometimes referred to as \emph{reductio ad absurdum}, e.g., in~\cite{odintsov2008constructive}). This system was introduced in this form by Johansson~\cite{johansson1937minimalkalkul} in 1937 by discarding \emph{ex falso quodlibet} from the standard axioms for intuitionistic logic. The system proposed by Johansson has its roots in Kolmogorov's formalization of intuitionistic logic~\cite{kolmogorov1925principle}. The axiomatization proposed by Johansson preserves the whole positive fragment and most of the negative fragment of Heyting's intuitionistic logic. As a matter of fact, many important properties of negation provable in Heyting's system remain provable (in some cases, in a slightly weakened form) in minimal logic. 

In this work, we focus on propositional logical systems arising from the language of minimal logic and obtained by weakening the requirements for the negation operator in a `maximal way'. More precisely, the bottom element of the bounded lattice of logics considered here is the system where the unary operator $\neg$ (that we still call `negation') does not satisfy any conditions except for being functional. The top element of this lattice is minimal logic. We use the term \emph{N-logic} to denote an arbitrary logical system in this lattice. This  setting is paraconsistent, in the sense that contradictory theories do not necessarily contain all formulas. 

In this paper we continue the study of N-logics started in~\cite{Colacito:Thesis:2016,colacito2016subminimal}. 
We investigate these logics  from several different perspectives. In Section~\ref{s:prelim} we give an algebraic and model-theoretic presentation of N-logics, and provide a brief recap of the main duality and completeness results from~\cite{Colacito:Thesis:2016}. We also review the semantics introduced for the N-logics in~\cite{Colacito:Thesis:2016,colacito2016subminimal}---that we call N-semantics---and show that it is a variant of standard neighbourhood semantics.

In Section~\ref{s:continuum}, we exploit these results to show that the lattice of N-logics has the cardinality of the continuum. The proofs from this section are obtained by exporting and adapting techniques of~\cite{dhjdejongh,jankov1968,bezhanishvili2006lattices}. 

Section~\ref{s:filtration} is devoted to a study of the method of  filtration for the basic N\nbd{-}logic. Model-theoretic and algebraic definitions of filtration are introduced and compared. 
This leads to the finite model property of the minimal logic.

Finally, Section~\ref{s:modal} concludes the article by introducing (bi-)modal systems that are proved to play the role of modal companions for N-logics. More precisely, the language of these systems contains a normal (namely, {\sf S4}) modality resulting from the positive fragment of intuitionistic logic, and a non-normal modality resulting from the negation operator.
After characterizing these logics in terms of standard neighbourhood semantics, we continue using the equivalent N\nbd{-}semantics. We then give a proof that the standard G\"odel translation of intuitionistic logic into {\sf S4} can be extended to translations of certain N-logics into these bi-modal systems. To the best of our knowledge this is the first use of non-normal modal operators in the context of the G\"odel translation. On the other hand, modal systems using a mix of normal and non-normal modalities have been recently explored in the evidence-based semantics of epistemic logic~\cite{vBP11,aybuke17}.


\section{Preliminaries}\label{s:prelim}

In this preliminary section we present the main technical tools that will be used throughout the paper. We start with a brief introduction to the Kripke semantics of minimal logic (here called \emph{N-semantics}), in line with the tradition of intuitionistic logic. Later we present the algebraic semantics, and state some basic facts. In order to keep the structure of the paper as simple as possible, we 
skip  the broader and introductory account of the topic  and refer the interested reader to~\cite{Colacito:Thesis:2016,colacito2016subminimal}. For a proof-theoretic account, see~\cite{BilkovaColacito2019}.

Let $\mathcal{L}({\sf Prop})$ be the propositional language, where ${\sf Prop}$ is a countable set of propositional variables, generated by the following grammar: 
\[
p \mid \top \mid \varphi \land \varphi \mid \varphi \lor \varphi \mid \varphi \to \varphi \mid \neg \varphi
\]  
\noindent 
where $p \in {\sf Prop}$. We omit $\bot$ from the language. We call a formula \emph{positive} if it contains only connectives from $\{\land, \lor, \to, \top\}$, and we refer to the positive fragment of intuitionistic logic as \emph{positive logic}. We start by considering a system defined by the axioms of positive logic, with the additional axiom $(p\leftrightarrow q)\to(\neg p\leftrightarrow\neg q)$ defining the behaviour of $\neg$. We call the resulting system {\sf N}. We fix the positive logical fragment, and we strengthen the negation operator up to reaching minimal propositional logic, which can be seen as the system obtained by adding the axiom $(p\to q) \land (p\to\neg q)\to \neg p$ to positive logic~\cite{rasiowa1974algebraic}. An alternative axiomatization of minimal logic is obtained by extending {\sf N} with the axiom $(p\to\neg p)\to\neg p$~\cite[Proposition 1.2.5]{Colacito:Thesis:2016}. 

If we interpret $\neg$ as a \textquoteleft{modality}', and disregard the fact that we consider extensions of positive logic, the basic system {\sf N} can be seen as an extension of classical modal logic~(see~\cite{pacuit2017neighborhood}) which is based on the rule $p \leftrightarrow q\,/\,\square p \leftrightarrow\square q$. Thus, {\sf N} can be regarded as a weak intuitionistic modal logic that---as far as we know---has not been previously studied. Note that extending it with more of the properties of a negation would lead to \textquoteleft{very non-standard}' modal logics. This relationship with modal logic will be further clarified towards the end of the paper, where the negation will be interpreted by a full-fledged modal operator.

A first algebraic account of Johansson's logic can be found in Rasiowa's  work on non-classical logic~\cite{rasiowa1974algebraic}, where the algebraic counterpart of minimal logic is identified as the variety of contrapositionally complemented lattices. A contrapositionally complemented lattice is an algebraic structure $\langle A, \land, \lor, \to, \neg, 1\rangle$, where $\langle A, \land, \lor, \to, 1 \rangle$ is a relatively pseudo-complemented lattice (which algebraically characterizes positive logic~\cite{rasiowa1974algebraic}) and the unary fundamental operation $\neg$ satisfies the identity $(x \to \neg y) \approx (y \to \neg x)$. The variety presented by Rasiowa is term-equivalent to the variety of relatively pseudo-complemented lattices with a negation operator defined by the algebraic version of the principle of contradiction $(x \to y) \land (x \to \neg y) \to \neg x \approx 1$, originally employed in Johansson's axiomatization. Observe that Heyting algebras can be seen as contrapositionally complemented lattices where $\neg 1$ is a distinguished bottom element $0$. 

We further generalize the notion of Heyting algebra to that of an N\nbd{-}algebra. An \emph{N-algebra} is an algebraic structure $\mathbf{A} = \langle A,\land,\lor,\to,\neg,1\rangle$, where $\langle A,\land,\lor,\to,1\rangle$ is a relatively pseudo-completemented lattice and $\neg$ is a unary operator satisfying the identity $(x\leftrightarrow y)\to(\neg x\leftrightarrow\neg y)\approx 1$. The latter can be equivalently formulated as $x\land\neg y\approx x\land \neg(x\land y)$. Note that this variety plays a fundamental role in the attempts of defining a connective over positive logic. In fact, the considered equation states that the function $\neg$ is a \emph{compatible function} (or \emph{compatible connective}), in the sense that every congruence of $\langle A, \land, \lor, \to, 1 \rangle$ is a congruence of $\langle A, \land, \lor, \to, \neg, 1 \rangle$. This is somehow considered a minimal requirement when introducing a new connective over a fixed setting (see, e.g.,~\cite{caicedo2001algebraic,ertola2007compatible}). Clearly, with every N-logic {\sf L} we can associate a variety of N-algebras. Each of these logics is complete with respect to its algebraic semantics~\cite{Colacito:Thesis:2016}. Contrapositionally complemented lattices are the strongest structures that we consider here, and can be seen as the variety of N-algebras defined by the equation $(x \to \neg x) \to \neg x \approx 1$, or $x \to \neg x \approx \neg x$. We also consider the two varieties of N-algebras defined, respectively, by the identity $(x \land \neg x) \to \neg y \approx 1$, and by the identity $(x \to y) \to (\neg y \to \neg x) \approx 1$. They were studied in detail in~\cite{Colacito:Thesis:2016}, and we shall refer to the corresponding logics as \emph{negative ex falso} logic ({\sf NeF}) 
and \emph{contraposition} logic ({\sf CoPC}). We point out that the logic {\sf CoPC} has appeared before under the name \textquoteleft{Subminimal Logic}' with a completely different semantics~(\cite{Haz92,Haz95}, \cite[Section 8.33]{Hum11}). It was proved in~\cite{Colacito:Thesis:2016,colacito2016subminimal} that the following relations hold between the considered logical systems:
\[
{\sf N} \subset {\sf NeF} \subset {\sf CoPC} \subset {\sf MPC},
\]
where the strict inclusion ${\sf L}_1 \subset {\sf L}_2$ means that every theorem of ${\sf L}_1$ is a theorem of ${\sf L}_2$, but there is at least one theorem of ${\sf L}_2$ that is not provable in ${\sf L}_1$. Note the strict inclusion ${\sf NeF} \subset {\sf CoPC}$ can be seen semantically (through the semantics we introduce below), by taking $\langle W, \leq, {\rm N}\rangle$, with $W \coloneqq \{w, v\}$, $\leq \,\coloneqq \{(w,v)\}$, and ${\rm N}(\emptyset) = {\rm N}(W) \coloneqq \{v\}$, ${\rm N}(\{v\}) \coloneqq W$. 

Recall that an \emph{intuitionistic Kripke frame} $\mathfrak{F}$ is a partially ordered set (briefly, poset) $\langle W, \leq \rangle$, and a Kripke model is a frame $\mathfrak{F}$ equipped with a valuation $V$ assigning to every propositional variable $p \in {\sf Prop}$ an upward closed subset (\emph{upset}) $V(p) \in \mathcal{U}(W,\leq)$ of $\mathfrak{F}$ where $p$ is true. In the subminimal setting, we call an {\em N-frame} (sometimes we may call it just a {\em frame}) a triple $\mathfrak{F} = \langle W, \leq, {\rm N} \rangle$, where $\langle W, \leq \rangle$ is a poset and the function ${\rm N} \colon \mathcal{U}(W, \leq) \to \mathcal{U}(W, \leq)$ satisfies, for every $X, Y \in \mathcal{U}(W, \leq)$,\footnote{This property is equivalent to $w \in {\rm N}(X) \Longleftrightarrow w \in {\rm N}(X \cap R(w))$ for every $w \in W$ and $X \in \mathcal{U}(W,\leq)$; see, e.g., \cite[Lemma 4.3.1]{Colacito:Thesis:2016}.}
\begin{equation}\label{eq:locality}
{\rm N}(X) \cap Y = {\rm N}(X \cap Y) \cap Y.
\end{equation}
An N-model (sometimes we may call it just a {\em model}) is again a pair $\langle \mathfrak{F}, V \rangle$, where $\mathfrak{F}$ is a frame and $V$  a valuation from the set of propositional variables to $\mathcal{U}(W, \leq)$. Given a model $\langle \mathfrak{F}, V \rangle$, we define truth of positive formulas inductively as in the intuitionistic setting, and say that a negative formula $\neg\varphi$ is true at a node $w$, written $\langle \mathfrak{F}, V \rangle, w \models \neg\varphi$, if $w \in {\rm N}(V(\varphi))$, where $V(\varphi) \coloneqq \{w \in W \colon \mathfrak{M}, w \models \varphi\}$. We use the customary notation and write $\mathfrak{F} \models \varphi$ if $\langle \mathfrak{F}, V \rangle, w \models \varphi$ holds for every $w \in W$ and every valuation $V$. Note that $\mathfrak{F} \models (p \land \neg p) \to \neg q$ holds on a frame satisfying property~(\ref{eq:locality}) if and only if $X \cap {\rm N}(X) \subseteq {\rm N}(Y)$ for arbitrary upsets $X, Y$, and $\mathfrak{F} \models (p \to q) \to (\neg q \to \neg p)$ holds if and only if the function ${\rm N}$ is antitone (i.e., $X \subseteq Y$ implies ${\rm N}(Y) \subseteq {\rm N}(X)$). The least element of a poset, when it exists, is called a \emph{root}, and we say that a frame is \emph{rooted} if its underlying poset has a root. Given a poset $\langle W, \leq \rangle$, we use $R(w)$ to denote the set $\{v \in W \colon w \leq v\}$ of successors of $w$ in $W$. In the rest of the paper, we refer to this \textquoteleft Kripke style' semantics for subminimal logics as {{\rm N}-semantics}.

For a frame $\mathfrak{F} = \langle W, \leq, {\rm N} \rangle$, we define the corresponding neighbourhood frame $\mathfrak{F_n} = \langle W, \leq, \mathfrak{n} \rangle$ by setting a neighbourhood function $\mathfrak{n} \colon W \to \mathcal{P}(\mathcal{U}(W, \leq))$ to be $\mathfrak{n}(w)\coloneqq\{X\subseteq W \mid w\in {\rm N}(X)\}$. A model is again a pair $\langle \mathfrak{F}, V \rangle$ consisting of a neighbourhood frame, and a valuation $V$ defined in the same way as for the N-semantics. Then, $\langle \mathfrak{F_n}, V \rangle, w \models \neg\varphi$ if and only if $V(\varphi)\in \mathfrak{n}(w)$ or, equivalently, $w \in {\rm N}(V(\varphi))$, that is, $\langle \mathfrak{F}, V \rangle, w \models \neg\varphi$. Conversely, take a neighbourhood frame $\mathfrak{F_n}=\langle W, \leq, \mathfrak{n} \rangle$ such that $\mathfrak{n} \colon W \to \mathcal{P}(\mathcal{U}(W, \leq))$ is a monotone function (i.e.,\ if $w\leq v$ then $\mathfrak{n}(w)\subseteq \mathfrak{n}(v)$) satisfying $X \in \mathfrak{n}(w)$ if and only if $X\cap R(w) \in \mathfrak{n}(w)$ for every upset $X$. The triple $\mathfrak{F}=\langle W, \leq, {\rm N}\rangle$ with ${\rm N}(X)\coloneqq\{w\in W\mid X\in \mathfrak{n}(w)\}$ is an N-frame, and the two definitions are mutually inverse. It will become clear in Section~\ref{s:modal} that an approach using the N-semantics is helpful in practice. Further, the N-semantics is synergic to the algebraic approach, as the next paragraph shows.

For an N-algebra $\mathbf{A}$, we consider the set $W_{\mathbf{A}}$ of prime filters of $\mathbf{A}$, and let $\widehat{a} \coloneqq \{w\in W_{\mathbf{A}}\mid a\in w\}$.  
Then the triple $\mathfrak{F}_{\mathbf{A}}=\langle W_{\mathbf{A}},\subseteq,{\rm N}_{\mathbf{A}} \rangle$ is a frame, where 
\[
{\rm N}_{\mathbf{A}}(X)\coloneqq\{w\in W_{\mathbf{A}}\,|\,(\exists \neg a\in w)({R(w)}\cap \widehat{a}={R(w)}\cap X)\},
\] 
for any upset $X\in\mathcal{U}(W_{\mathbf{A}}, \subseteq)$. This definition makes sure that ${\rm N}_{\mathbf{A}}(\widehat{b})$ for $b \in \mathbf{A}$ includes every filter $w$ which contains $\neg a$, for any $a \in \mathbf{A}$ that is equivalent to $b$ with respect to $w$ (i.e., ${R(w)}\cap \widehat{a}={R(w)}\cap \widehat{b}$). Observe that the notion of prime filter in this context does not require the filter to be proper, i.e., the whole algebra $A$ is always a prime filter. On the other hand, starting from a frame $\mathfrak{F}$, we obtain an N-algebra $\mathbf{A}_{\mathfrak{F}} = \langle \mathcal{U}(W,\leq),\cap,\cup,\to,{\rm N},W \rangle$ whose universe is the set of upsets of $\mathfrak{F}$ equipped with the usual intersection, union, Heyting implication, and unit $W$, and with the unary operator given by the function ${\rm N}$. Note that the N-algebra $\mathbf{A}$ embeds into the N-algebra $\mathbf{A}_{\mathfrak{F}_{\mathbf{A}}}$ via the map $\alpha\colon a\mapsto\widehat{a}$. A consequence of this is that each valuation $\mu \colon {\sf Prop} \to \mathbf{A}$ on $\mathbf{A}$ gives rise to a valuation $V=\alpha\circ \mu$ on $\mathfrak{F}_{\mathbf{A}}$. 

In order to obtain a full duality result between algebraic and frame-theoretic structures, in analogy with intuitionistic logic we introduce general frames. A \emph{general frame} is a quadruple $\mathfrak{F}=\langle W,\leq,\mathcal{P},{\rm N}\rangle$, where $\langle W,\leq\rangle$ is a partially ordered set, $\mathcal{P} \subseteq \mathcal{U}(W,\leq)$ contains $W$ and is closed under $\cup$, $\cap$, Heyting implication $\to$, and ${\rm N}\colon \mathcal{P}\to\mathcal{P}$ satisfies for all $X,Y\in\mathcal{P}$, ${\rm N}(X) \cap Y={\rm N}(X\cap Y) \cap Y$. Elements of $\mathcal{P}$ are called \emph{admissible sets}. Note that 
$\emptyset$ need not be admissible. We next recall the definition of \emph{refined} and \emph{compact} general frames, see e.g.,~\cite{CZ97} for analogous definitions in the intuitionistic setting. If a general frame $\mathfrak{F}$ is refined and compact, we call it a \emph{descriptive frame}. Finally, a \emph{top descriptive frame} $\mathfrak{F} = \langle W,\leq,\mathcal{P},{\rm N}, t\rangle$ is a descriptive frame whose partially ordered underlying set $\langle W,\leq\rangle$ has a greatest element $t$ included in every admissible upset $X\in\mathcal{P}$. Following~\cite{bezhanishviliuniversal}, we call \emph{top model} a pair $\langle \mathfrak{F}, V \rangle$ where $\mathfrak{F}$ is a top descriptive frame and $V$ is a valuation on $\mathfrak{F}$ which makes every propositional variable true at the top node. Every finite frame with a top node is a top descriptive frame with a set of admissible upsets $\{X \in \mathcal{U}(W,\leq) \colon X \ne \emptyset\}$.

It can be proved~\cite{Colacito:Thesis:2016,bezhanishvili2017epimorphisms} that the correspondence between $\mathfrak{F}$ and $\mathbf{A}_{\mathfrak{F}}$, and between $\mathbf{A}$ and $\mathfrak{F}_{\mathbf{A}}$ that holds for every frame and N-algebra, in the case of top descriptive frames gives rise to a proper duality. More precisely, given a top descriptive frame $\mathfrak{F} = \langle W,\leq,\mathcal{P},{\rm N}, t\rangle$, the structure $\mathbf{A}_{\mathfrak{F}}=\langle \mathcal{P},\cap,\cup,\to,{\rm N},W \rangle$ is the N-algebra dual to $\mathfrak{F}$, and the set of prime filters of any N-algebra $\mathbf{A} = \langle A,\land,\lor,\to,\neg,1\rangle$ induces a dual top descriptive frame $\mathfrak{F}_{\mathbf{A}} = \langle W_{\mathbf{A}}, \subseteq, \mathcal{P}_{\mathbf{A}}, {\rm N}_{\mathbf{A}}, A \rangle$, where $\mathcal{P}_{\mathbf{A}}$ is the set $\{\hat{a} \colon a \in A\}$, and the map ${\rm N}_{\mathbf{A}}$ is defined as ${\rm N}_{\mathbf{A}}(\widehat{a})=\widehat{(\neg a)}$. Moreover, we have $\mathfrak{F} \cong \mathfrak{F}_{\mathbf{A}_{\mathbf{F}}}$ and $\mathbf{A} \cong \mathbf{A}_{\mathfrak{F}_{\mathbf{A}}}$. Observe that the fact that prime filters of $\mathbf{A}$ do not need to be proper ensures that the corresponding frame structure has a top element. 

As in the case of Heyting algebras, for N-algebras there exists a one-to-one correspondence between congruences and filters. We can therefore characterize subdirectly irreducible N-algebras as those N-algebras containing a second greatest element, or equivalently, as those N-algebras $\mathbf{A}$ whose dual frame $\mathfrak{F}_{\mathbf{A}}$ is a rooted top descriptive frame. We call a top descriptive frame $\mathfrak{F}$ {finitely generated} if $\mathbf{A}_{\mathfrak{F}}$ is a finitely generated N-algebra, thereby obtaining a completeness result for every N-logic with respect to the class of its finitely generated rooted top descriptive frames.


\section{Continuum Many Logics}\label{s:continuum}

In this section, we construct continuum many N-logics by exporting and adapting techniques of~\cite{dhjdejongh,jankov1968,bezhanishvili2006lattices}. We start from a countable family of positive formulas that we adapt so to define uncountably many independent subsystems of minimal logic containing the basic logic {\sf N}. 

Given two finite rooted posets $\mathfrak{F}, \mathfrak{G}$, we write $\mathfrak{F} \leq \mathfrak{G}$ if $\mathfrak{F}$ is an order-preserving image of $\mathfrak{G}$, that is, if there is an onto map $f \colon \mathfrak{G} \twoheadrightarrow \mathfrak{F}$ which preserves the order. This relation can be proved to be a partial order on the set of finite (rooted) posets~\cite{bezhanishvili2006lattices}. Consider the sequence ${\rm \Delta}=\{\mathfrak{F}_n\mid n \in \mathbb{N}\}$ of finite rooted posets with a top node (Figure~\ref{fig:1}). The sequence ${\rm \Delta}$ is obtained from the antichain presented in~\cite[Figure 2]{bezhanishvili2017locally} by adding three nodes on the top. An argument resembling the one in~\cite[Lemma 6.12]{bezhanishvili2017locally} shows that the sequence ${\rm \Delta}$ is a $\leq$-antichain. We quickly sketch it here. If $f \colon \mathfrak{F}_n \twoheadrightarrow \mathfrak{F}_m$ with $n \geq m$ is order-preserving, then the upset of $\mathfrak{F}_m$ consisting of points of depth $\leq (m+1)$ must be the image of the upset of $\mathfrak{F}_n$ consisting of points of depth $\leq (m+1)$. Thus, for $x_{m+2} \in \mathfrak{F}_m$, there is $y \in \mathfrak{F}_n$ such that $f(y)=x_{m+2}$. But then, there is a $z \in \mathfrak{F}_n$ such that $y \leq z$ and $f(z)=y_{m+1}$. Since $x_{m+2} \not\leq y_{m+1}$, we get a contradiction.

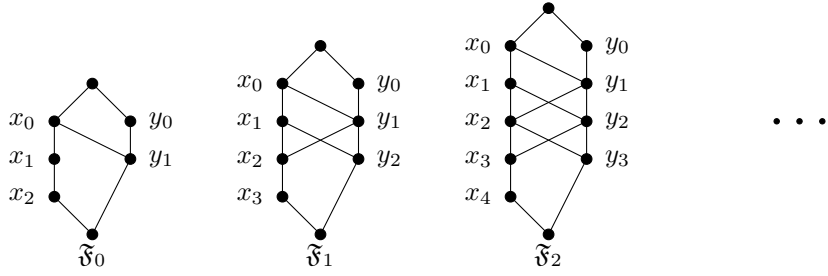
\begin{figure}[!]
\centering
\begin{tikzpicture}
\filldraw[black] (-3,0) circle (2pt) node[anchor=north] {$\mathfrak{F}_0$};
\filldraw[black] (-3.5,0.5) circle (2pt) node[anchor=east] {$x_2\,\,$};
\filldraw[black] (-3.5,1) circle (2pt) node[anchor=east] {$x_1\,\,$};
\filldraw[black] (-3.5,1.5) circle (2pt) node[anchor=east] {$x_0\,\,$};
\filldraw[black] (-3,2) circle (2pt) node[anchor=south] {};
\filldraw[black] (-2.5,1) circle (2pt) node[anchor=west] {$\,\,y_1$};
\filldraw[black] (-2.5,1.5) circle (2pt) node[anchor=west] {$\,\,y_0$};
\draw[black] (-3,0) -- (-3.5,0.5);
\draw[black] (-3.5,0.5) -- (-3.5,1);
\draw[black] (-3.5,1) -- (-3.5,1.5);
\draw[black] (-3.5,1.5) -- (-3,2);
\draw[black] (-3,0) -- (-2.5,1);
\draw[black] (-2.5,1) -- (-2.5,1.5);
\draw[black] (-2.5,1) -- (-3.5,1.5);
\draw[black] (-2.5,1.5) -- (-3,2);
\filldraw[black] (0,0) circle (2pt) node[anchor=north] {$\mathfrak{F}_1$};
\filldraw[black] (-0.5,0.5) circle (2pt) node[anchor=east] {$x_3\,\,$};
\filldraw[black] (-0.5,1) circle (2pt) node[anchor=east] {$x_2\,\,$};
\filldraw[black] (-0.5,1.5) circle (2pt) node[anchor=east] {$x_1\,\,$};
\filldraw[black] (-0.5,2) circle (2pt) node[anchor=east] {$x_0\,\,$};
\filldraw[black] (0,2.5) circle (2pt) node[anchor=south] {};
\filldraw[black] (0.5,1) circle (2pt) node[anchor=west] {$\,\,y_2$};
\filldraw[black] (0.5,1.5) circle (2pt) node[anchor=west] {$\,\,y_1$};
\filldraw[black] (0.5,2) circle (2pt) node[anchor=west] {$\,\,y_0$};
\draw[black] (0,0) -- (-0.5,0.5);
\draw[black] (-0.5,0.5) -- (-0.5,1);
\draw[black] (-0.5,1) -- (-0.5,1.5);
\draw[black] (-0.5,1.5) -- (-0.5,2);
\draw[black] (-0.5,2) -- (0,2.5);
\draw[black] (0,0) -- (0.5,1);
\draw[black] (0.5,1) -- (0.5,1.5);
\draw[black] (0.5,1.5) -- (0.5,2);
\draw[black] (0.5,2) -- (0,2.5);
\draw[black] (-0.5,1) -- (0.5,1.5);
\draw[black] (0.5,1) -- (-0.5,1.5);
\draw[black] (0.5,1.5) -- (-0.5,2);
\filldraw[black] (3,0) circle (2pt) node[anchor=north] {$\mathfrak{F}_2$};
\filldraw[black] (2.5,0.5) circle (2pt) node[anchor=east] {$x_4\,\,$};
\filldraw[black] (2.5,1) circle (2pt) node[anchor=east] {$x_3\,\,$};
\filldraw[black] (2.5,1.5) circle (2pt) node[anchor=east] {$x_2\,\,$};
\filldraw[black] (2.5,2) circle (2pt) node[anchor=east] {$x_1\,\,$};
\filldraw[black] (2.5,2.5) circle (2pt) node[anchor=east] {$x_0\,\,$};
\filldraw[black] (3,3) circle (2pt) node[anchor=south] {};
\filldraw[black] (3.5,1) circle (2pt) node[anchor=west] {$\,\,y_3$};
\filldraw[black] (3.5,1.5) circle (2pt) node[anchor=west] {$\,\,y_2$};
\filldraw[black] (3.5,2) circle (2pt) node[anchor=west] {$\,\,y_1$};
\filldraw[black] (3.5,2.5) circle (2pt) node[anchor=west] {$\,\,y_0$};
\draw[black] (3,0) -- (2.5,0.5);
\draw[black] (2.5,0.5) -- (2.5,1);
\draw[black] (2.5,1) -- (2.5,1.5);
\draw[black] (2.5,1.5) -- (2.5,2);
\draw[black] (2.5,2) -- (2.5,2.5);
\draw[black] (2.5,2.5) -- (3,3);
\draw[black] (3,0) -- (3.5,1);
\draw[black] (3.5,1) -- (3.5,1.5);
\draw[black] (3.5,1.5) -- (3.5,2);
\draw[black] (3.5,2) -- (3.5,2.5);
\draw[black] (3.5,2.5) -- (3,3);
\draw[black] (2.5,1) -- (3.5,1.5);
\draw[black] (3.5,1) -- (2.5,1.5);
\draw[black] (3.5,1.5) -- (2.5,2);
\draw[black] (2.5,1.5) -- (3.5,2);
\draw[black] (3.5,2) -- (2.5,2.5);
\filldraw[black] (6,1.5) circle (1pt) node[anchor=west] {};
\filldraw[black] (6.3,1.5) circle (1pt) node[anchor=west] {};
\filldraw[black] (6.6,1.5) circle (1pt) node[anchor=west] {};
\end{tikzpicture} 
\caption{The sequence ${\rm \Delta}$}
\label{fig:1}
\end{figure}

\begin{lemma}
The sequence ${\rm \Delta}$ forms a $\leq$-antichain.
\end{lemma}

\noindent We call a map $f\colon \mathfrak{F} \to \mathfrak{G}$ between posets $\mathfrak{F} = \langle W, \leq \rangle$ and $\mathfrak{G} = \langle W', \leq' \rangle$ a p-morphism if $f$ is order-preserving, and $f(w) \leq' v$ implies the existence of $w' \in W$ such that $w \leq w'$ and $f(w') = v$. A partial p-morphism such that $\mathrm{dom}(f)$ is a downward closed set is called a \emph{positive morphism}. We write $\mathfrak{F}\preceq \mathfrak{G}$ if $\mathfrak{F}$ is a positive morphic image of $\mathfrak{G}$. Assume $\mathfrak{F}\preceq \mathfrak{G}$ and let $f$ be a positive morphism from 
$\mathfrak{G}$ onto $\mathfrak{F}$. Extending $f$ by mapping all the points of $\mathfrak{G} \setminus \mathrm{dom}(f)$ to the top node of $\mathfrak{F}$, we obtain a total order-preserving map, yielding $\mathfrak{F}\leq \mathfrak{G}$. This ensures the $\leq$-antichain ${\rm \Delta}$ to be also a $\preceq$-antichain.

\begin{lemma}
If $\mathfrak{F}\preceq \mathfrak{G}$ and $\mathfrak{F}$ is a finite rooted posets with a top node, then $\mathfrak{F}\leq \mathfrak{G}$.
\end{lemma}

\noindent Having constructed the desired $\preceq$-antichain, we now proceed by adjusting the technique of Jankov-de Jongh formulas. 
Recall that the universal model $\mathcal{U}(n)$ can be roughly thought of as the upper generated submodel of the descriptive model dual to the free Heyting algebra on $n \in \mathbb{N}$ generators. The universal model $(\mathcal{U}^{\star}(n))^{+}$ for positive logic is isomorphic to a generated submodel of $\mathcal{U}(n)$, and is a top model~\cite{bezhanishvili2006lattices,bezhanishviliuniversal}. It was shown in~\cite{bezhanishviliuniversal} that every finite rooted poset with a top node, equipped with an appropriate valuation, is isomorphic to a generated submodel of $(\mathcal{U}^{\star}(m))^{+}$ for some $m$. Note that, for the frames $\mathfrak{F}_n$ in Figure~\ref{fig:1}, it suffices to take $m=3$. Hence, we call $w_n\in(\mathcal{U}^{\star}(3))^{+}$ the node in the positive universal model corresponding to the root of $\mathfrak{F}_n$. We can assume without loss of generality the positive Jankov-de Jongh formula of $\mathfrak{F}_n$ to be defined as follows~\cite{bezhanishviliuniversal}: 
\[
\chi^{\star}(\mathfrak{F}_n) = \psi^{\star}_{w_n}\coloneqq \varphi^{\star}_{w_n} \to \bigvee_{i=1}^{r}\varphi^{\star}_{w_{n_{i}}}\, ,
\]
\noindent where $\varphi^{\star}_{w_n}, \varphi^{\star}_{w_{n_{i}}}$ are defined as in~\cite{bezhanishviliuniversal} and $w_n \prec \{w_{n_{1}},\ldots,w_{n_{r}}\}$. So, a rooted poset $\mathfrak{G}$ refutes $\chi^{\star}(\mathfrak{F}_n)$ if and only if $\mathfrak{F}_n\preceq\mathfrak{G}$. As ${\rm \Delta}$ is a $\preceq$-antichain, this means that, for every $n,m \in \mathbb{N}$, the formula $\chi^{\star}(\mathfrak{F}_m)$ is valid on the poset $\mathfrak{F}_n$ if and only if $n \ne m$. In fact, the construction of $\varphi^{\star}_{w_n}, \varphi^{\star}_{w_{n_{i}}}$ ensures that the formula $\varphi^{\star}_{w_n}$ is satisfied at the root $w_n$ in $(\mathcal{U}^{\star}(2))^{+}$, while none of the formulas $\varphi^{\star}_{w_{n_{i}}}$ is.  

\begin{lemma}
For $n,m \in \mathbb{N}$, the formula $\mathfrak{F}_n \models \chi^{\star}(\mathfrak{F}_m)$ if and only if $n \ne m$.
\end{lemma}

\noindent At this point, we are going to equip each rooted poset $\mathfrak{F}_n$ with an appropriate function ${\rm N}_n$ to make it a top descriptive frame. Moreover, we enhance the definition of $\chi^{\star}(\mathfrak{F}_n)$ in order to obtain formulas $\theta(\mathfrak{F}_n)$ with the same defining property of the Jankov-de Jongh formulas for the signature of positive logic, but with the extra addition that the formulas $\theta(\mathfrak{F}_n)$ are theorems of 
{\sf MPC}. Given the rooted poset $\mathfrak{F}_n$, we consider the top descriptive frame $\langle \mathfrak{F}_n, {\rm N}_n, t_n \rangle$ where ${\rm N}_n$ has the property ${\rm N}_n(\{t_n\})=\{t_n\}$, the element $t_n$ being the top node of $\mathfrak{F}_n$. We denote the new family of top descriptive frames $\langle \mathfrak{F}_n, {\rm N}_n, t_n \rangle$ by ${\rm \Delta}_{N}$. Now we consider a fresh propositional variable $p$, and define: 
\[
\theta(\mathfrak{F}_n) \coloneqq (p \to\neg p) \land \varphi^{\star}_{w_n} \to \neg p \lor \bigvee_{i=1}^{r}\varphi^{\star}_{w_{n_{i}}}\, .
\]
\noindent 
We note that the formulas $\theta(\mathfrak{F}_n)$ are not the Jankov-de Jongh formulas for the considered signature. It is nonetheless easy to see that, if $n \ne m$, the formula $\theta(\mathfrak{F}_n)$ is valid on the frame $\langle \mathfrak{F}_m, {\rm N}_m, t_m \rangle$. On the other hand, for checking that $\langle \mathfrak{F}_n, {\rm N}_n, t_n \rangle \not\vDash \theta(\mathfrak{F}_n)$ it is enough to consider a valuation $\tilde{V}_n$ extending $V_n$ by  $\tilde{V}_n(p)=\{t_n\}$. This way, the root of $\mathfrak{F}_n$ under the considered valuation makes $(p \to\neg p) \land \varphi^{\star}_{w_n}$ true, while neither $\neg p$ nor $\bigvee_{i=1}^{r}\varphi^{\star}_{w_{n_{i}}}$ is true at $w_n$.

As a consequence, we obtain the following result.

\begin{theorem}\label{t:continuum}
There are continuum many N-logics. 
\end{theorem}

\begin{proof}
The fact that each $\theta(\mathfrak{F}_n)$ is a theorem of {\sf MPC} ensures that for each subset ${\rm \Gamma}\subseteq {\rm \Delta}_N$, the logic  
\[
L({\rm \Gamma}) \coloneqq \mathsf{N} + \{\theta(\mathfrak{F})\mid \mathfrak{F} \in {\rm \Gamma}\}
\] 
belongs to the interval $[{\sf N}, {\sf MPC}]$. Finally, observe that for each pair of different subsets ${\rm \Gamma}_1 \ne {\rm \Gamma}_2$ of ${\rm \Delta}_N$, we have $L({\rm \Gamma}_1) \ne L({\rm \Gamma}_2)$. Indeed, without loss of generality we may assume that there is $\mathfrak{F} \in {\rm \Gamma}_1$ such that  $\mathfrak{F} \notin {\rm \Gamma}_2$. Moreover, we have $\mathfrak{F} \not\vDash \theta(\mathfrak{F})$ and $\mathfrak{F} \vDash \theta(\mathfrak{G})$, for each $\mathfrak{G}$ in ${\rm \Gamma}_2$. Therefore, there is a top descriptive frame $\mathfrak{F}$ which is an $L({\rm \Gamma}_2)$-frame and not an  $L({\rm \Gamma}_1)$-frame. Since every N-logic is complete with respect to top descriptive frames, the latter entails that $L({\rm \Gamma}_1) \ne L({\rm \Gamma}_2)$.   
\end{proof}

Observe that the choices of ${\rm N}_n$ and of the formula $(p \to \neg p) \to \neg p$ that were made above are arbitrary among the ones that ensure a proof of Theorem~\ref{t:continuum}. This suggests the question what would have happened if we had made different choices. As a matter of fact, the result of Theorem~\ref{t:continuum} can be refined as follows. 

\begin{proposition}\label{t:continuumnarrow}\
\begin{newlist}
\item[\rm (1)] There are continuum many N-logics between {\sf NeF} and {\sf CoPC}. \smallskip
\item[\rm (2)] There are continuum many N-logics below {\sf NeF}. \smallskip
\item[\rm (3)] There are continuum many N-logics above {\sf CoPC}.
\end{newlist}
\end{proposition}

\begin{proof}
The idea is to slightly modify the previous proof by choosing different theorems of minimal propositional logic, and different functions ${\rm N}_n$. Changing the former allows us to refine the result from above, while changing the latter from below. 

To prove (1), we consider posets $\mathfrak{F}_n$, but this time we consider the function ${\rm N}^1_n$ defined by
\[
{\rm N}^1_n(X) \coloneqq 
\begin{cases}
W_n \setminus \{w_n\}, & \text{ if }X \ne W_n \setminus \{w_n\} \\
W_n, & \text{ if }X = W_n \setminus \{w_n\}.
\end{cases}
\]
Let $p, q$ be fresh propositional variables. Consider the formulas
\[
\theta^1(\mathfrak{F}_n)\coloneqq (p \to q) \land \varphi^{\star}_{w_n} \to (\neg q \to \neg p) \lor \bigvee_{i=1}^{r}\varphi^{\star}_{w_{n_{i}}}\,.
\]
\noindent 
Observe that the frame $\langle \mathfrak{F}_n, {\rm N}^1_n, t_n \rangle$ is a frame of the logic $\mathsf{NeF}$. In fact, if we consider an arbitrary $w \in W_n$ such that $w \in X \cap {\rm N}^1_n(X)$ (this amounts to every $w \in W_n \setminus \{w_n\}$), then $w \in {\rm N}^1_n(Y)$ for every admissible upset $Y$. Moreover, it is easy to see that, if $n \ne m$, the formula $\theta^1(\mathfrak{F}_n)$ is valid on the frame $\langle \mathfrak{F}_m, {\rm N}^1_m, t_m \rangle$. On the other hand, for checking that $\langle \mathfrak{F}_n, {\rm N}^1_n, t_n \rangle \not\vDash \theta^1(\mathfrak{F}_n)$ it is enough to consider a valuation $\tilde{V}_n$ extending 
$V_n$ by $\tilde{V}_n(p)=\{t_n\}$ and $\tilde{V}_n(q)=W_n \setminus \{w_n\}$. This way, the root of $\mathfrak{F}_n$ under the considered valuation makes the whole antecedent of $\theta^1(\mathfrak{F}_n)$ true (since $\tilde{V}_n(p) \subseteq \tilde{V}_n(q)$), while the consequent is not true at $w_n$ (since $w_n \in {\rm N}^1_n(\tilde{V}_n(q))$, while $w_n \not \in {\rm N}^1_n(\tilde{V}_n(p))$). 

A proof of (2) is obtained by considering functions:
\[
{\rm N}^2_n(X) \coloneqq 
\begin{cases}
W_n \setminus \{w_n\}, & \text{ if }X \subset W, \\
W_n, & \text{ if }X = W,
\end{cases}
\]
and formulas 
\[
\theta^2(\mathfrak{F}_n)\coloneqq (p \land \neg p) \land \varphi^{\star}_{w_n} \to \neg q \lor \bigvee_{i=1}^{r}\varphi^{\star}_{w_{n_{i}}},
\] 
with fresh propositional variables $p, q$. 

For (3), it is enough to define ${\rm N}^3_n(X) \coloneqq \{t_n\}$, for every upset $X$, and to use the formulas $\theta(\mathfrak{F}_n)$.  
\end{proof}


\section{Filtration}\label{s:filtration}

The method of filtration is among the oldest and most used techniques for proving the finite model property for modal and superintuitionistic logics. If a model refutes a formula $\varphi$, then the idea is to filter it out (identify the points that agree on the subformulas of $\varphi$) to obtain a finite model that refutes $\varphi$. Thus the new model is a quotient of the original one. The method of filtration was originally developed algebraically~\cite{McK41,McKT44}, and later model-theoretically~\cite{Seg,Lem77}, see also~\cite{CZ97}. In this section, we complete the results and constructions from~\cite{Colacito:Thesis:2016}, where the finite model property was proved via model-theoretic filtrations applied to the canonical models only. A general model-theoretic definition of filtration was missing in~\cite{Colacito:Thesis:2016}. We give one here, which is a combination of the standard filtration for intuitionistic logic \cite{CZ97} 
and the one for non-normal logic \cite{chellas1980modal,pacuit2017neighborhood}. Building on the ideas of~\cite{bezhanishvilialgebraic,bezhanishvili2017locally}, we define algebraic filtrations for N-algebras. As in the modal and intuitionistic cases, we show, similarly to \cite{bezhanishvilialgebraic},  that the algebraic and model-theoretic methods turn out to be ``two sides of the same coin''.

Let ${\rm \Sigma}$ be a finite set of formulas closed under subformulas. Given a model $\mathfrak{M}$, we define an equivalence relation $\sim$ on $W$ by considering, for $w, v \in W$: 

\[
w\sim v\text{ if }(\forall\varphi\in{\rm \Sigma})(w\in V(\varphi)\Longleftrightarrow v\in V(\varphi)).
\]

\noindent Consider $\mathfrak{F}^{*} = \langle W^{*}, \leq^{*}, {\rm N}^{*} \rangle$, where $W^{*} = W/\sim$, and $\leq^{*}$ and ${\rm N}^{*}$ are, respectively, a partial order on $W^{*}$ and a function ${\rm N}^{*} \colon \mathcal{U}(W^{*}, \leq^{*}) \to \mathcal{U}(W^{*}, \leq^{*})$ satisfying the following conditions for all $w,v\in W$, $X\in\mathcal{U}(W^{*}, \leq^{*})$ and $\varphi, \neg\psi \in {\rm \Sigma}$: 

\begin{newlist}

\item [\rm (a)] $w\leq v\text{ implies }[w]\leq ^{*}[v]$; \smallskip

\item [\rm (b)] $[w]\leq ^{*}[v]\text{ and }w\in V(\varphi)\text{ imply }v\in V(\varphi)$; \smallskip

\item [\rm (c)] $[w] \in {\rm N}^{*}(X)\text{ implies }w \in {\rm N}(\pi^{-1}(X))$, 
\end{newlist}
where $\pi^{-1}(X) \coloneqq \{w\in W \colon [w]\in X\}$; and
\begin{newlist}

\item [\rm (d)] $w \in {\rm N}(V(\psi))\text{ implies }[w] \in {\rm N}^{*}(\pi[V(\psi)])$,

\end{newlist}
where $\pi[V(\psi)]\coloneqq\{[w] \in W^{*} \colon w \in V(\psi)\}$. Finally, let $V^{*}$ be a valuation on $\mathfrak{F}^{*}$ such that 
\[
V^{*}(p)\coloneqq\{[w] \in W^{*} \colon w\in V(p)\} = \pi[V(p)],
\]

\noindent for each $p\in{\rm \Sigma}$. We call $\mathfrak{M}^{*} = \langle \mathfrak{F}^{*},V^{*} \rangle$ a \emph{filtration} of the model $\mathfrak{M}$ through ${\rm \Sigma}$. 

\begin{theorem}[Filtration Theorem]
Let $\mathfrak{M}=\langle \mathfrak{F},V \rangle$ be a model. Consider the filtration $\mathfrak{M}^{*} = \langle \mathfrak{F}^{*},V^{*} \rangle$ of $\mathfrak{M}$ through a finite set ${\rm \Sigma}$ of formulas closed under subformulas. Then, for each $\varphi\in{\rm \Sigma}$ and $w\in W$ we have 
\[
w\in V(\varphi)\text{ iff }[w]\in V^{*}(\varphi).
\]
\end{theorem}

\begin{proof}
We focus on the case $\neg\varphi\in{\rm \Sigma}$. Consider $w\in V(\neg\varphi)$, i.e., $w\in {\rm N}(V(\varphi))$. By induction hypothesis, $V^{*}(\varphi)=\{[w] \colon w \in V(\varphi)\} = \pi[V(\varphi)]$. Hence, by (d), we are able to conclude $[w]\in {\rm N}^{*}(V^{*}(\varphi))$. On the other hand, assume $w\not\in {\rm N}(V(\varphi))$. By induction hypothesis, $V(\varphi)=\{w \colon [w] \in V'(\varphi)\} = \pi^{-1}(V^{*}(\varphi))$. Therefore, by (c), $[w]\not\in {\rm N}^{*}(V^{*}(\varphi))$.  
\end{proof}

Among the filtrations of $\mathfrak{M} = \langle W, \leq , {\rm N}, V \rangle$ through ${\rm \Sigma}$, there always exists a greatest filtration. Consider the partial order on $W^{*}$ defined by:
\begin{equation}\label{eq:biggest}
[w]\leq ^{g}[v]\text{ if } (\forall\varphi\in{\rm \Sigma})(w\in V(\varphi)\Rightarrow v\in V(\varphi)),
\end{equation}
and the map ${\rm N}^{g}(X) \coloneqq \{[w] \in W^{*} \colon w \in {\rm N}(\pi^{-1}(X))\}$. It is easy to check that $\mathfrak{F}^{\,g} = \langle W^{*}, \leq ^{g}, {\rm N}^{g} \rangle$ is a frame and $\mathfrak{M}^{\,g} = \langle \mathfrak{F}^{\,g}, V^{*} \rangle$ is a filtration of $\mathfrak{M}$. Consider now an arbitrary filtration $\mathfrak{M}^{*} = \langle \mathfrak{F}^{*},V^{*} \rangle$. Suppose $[w]\leq ^{*}[v]$. Then $w \in V(\varphi)$ implies $v \in V(\varphi)$. Moreover, observe that every upset of $\langle W^{*}, \leq ^{g} \rangle$ is an upset of $\langle W^{*}, \leq ^{*} \rangle$. Furthermore, for a 
given $X \in \mathcal{U}(W^{*}, \leq ^g)$ such that $[w] \in {\rm N}^{*}(X)$ we have  then $[w] \in {\rm N}^g(X)$. To see this, assume that $[w] \in {\rm N}^{*}(X)$. Then, by (c), $w \in {\rm N}(\pi^{-1}(X))$ and hence, $[w] \in {\rm N}^{g}(X)$.


The rest of this section will be devoted to an algebraic development of the notion of filtration. Let ${\rm \Sigma}$ be a finite set of formulas closed under subformulas, and consider an N-algebra with a valuation $\langle \mathbf{A}, \mu \rangle$, where $\mu \colon {\sf Prop} \to \mathbf{A}$. Since ${\rm \Sigma}$ is finite, so is $\mu[{\rm \Sigma}]$ as a subset of $\mathbf{A}$. Let $\mathbf{S}$ be the $(\land,\lor,1)$-reduct of $\mathbf{A}$ generated by $\mu[{\rm \Sigma}]$. Observe that $(\land,\lor,1)$-reducts are locally finite, i.e., every finitely generated $(\land,\lor,1)$-structure is finite. Hence, the resulting $\mathbf{S}$ is finite and we can make $\mathbf{S}$ an N-algebra by equipping it with operations $\to_{\mathbf{S}}$ and $\neg_{\mathbf{S}}$ defined as follows: 
\begin{equation}\label{eq:impl}
a\to_{\mathbf{S}} b\coloneqq\bigvee\{s\in S\mid a\land s\leq b\},
\end{equation}
and 
\begin{equation}\label{eq:neg}
\neg_{\mathbf{S}} a\coloneqq\bigvee\{s\in S\mid s\leq \neg a\},
\end{equation}
for every $a,b\in S$. Note that $\neg_{\mathbf{S}} a\leq \neg a$, and $\neg_{\mathbf{S}} a=\neg a$ whenever $\neg a\in \mu[{\rm \Sigma}]$, and $a \to_{\mathbf{S}} b\leq a \to b$ and $a \to_{\mathbf{S}} b = a \to b$ whenever $a\to b\in \mu[{\rm \Sigma}]$. We call $V$ the valuation $V=\alpha\circ \mu$ on $\mathfrak{F}_{\mathbf{A}}$, where $\alpha$ is defined as in Section~\ref{s:prelim}.

\begin{lemma}\label{l:filtrationalg}
The subset $S \subseteq A$ gives rise to a filtration $\mathfrak{M}^{*}_{\mathbf{A}}=\langle \mathfrak{F}_{\mathbf{A}},V^{*} \rangle$ of the model $\mathfrak{M}_{\mathbf{A}}=\langle \mathfrak{F}_{\mathbf{A}},V \rangle$ through ${\rm \Sigma}$.
\end{lemma}

\begin{proof}
Define $\sim$ on $W_{\mathbf{A}}$ by $w \sim v$ if and only if $S \cap w = S \cap v$. The first step amounts to proving that $w \sim v$ is equivalent to $w \cap \mu[{\rm \Sigma}] = v \cap \mu[{\rm \Sigma}]$. The proof follows the one of~\cite[Lemma 2.5]{bezhanishvilialgebraic}. Observe that this first step is clearly equivalent to $w \sim v$ if and only if $(\forall \varphi \in {\rm \Sigma})(\mu(\varphi) \in w \text{ if and only if } \mu(\varphi) \in v)$, which in turn is equivalent to $(\forall \varphi \in {\rm \Sigma})(w \in V(\varphi) \text{ if and only if } v \in V(\varphi))$. 

Now, let $W^{*}_{\mathbf{A}} = W_{\mathbf{A}} / \sim$ and define $[w] \leq ^{*} [v] \Longleftrightarrow w \cap S \subseteq v \cap S$, and 
\[
{\rm N}^{*}(X) \coloneqq \{[w] \in W^{*}_{\mathbf{A}} \colon w \in {\rm N}(\pi^{-1}(X))\},
\] 
for $X \in \mathcal{U}(W^{*}_{\mathbf{A}}, \leq ^{*})$. It is straightforward to see that $\mathfrak{F}^{*}_{\mathbf{A}} = \langle W^{*}_{\mathbf{A}}, \leq ^{*}, {\rm N}^{*} \rangle$ is a frame. Let now $V^{*}$ be a valuation on $\mathfrak{F}^{*}_{\mathbf{A}}$ such that: 
\[
V^{*}(p) \coloneqq \{[w] \in W^{*}_{\mathbf{A}} \colon w \in V(p)\}.
\]
The structure $\langle \mathfrak{F}^{*}_{\mathbf{A}}, V^{*} \rangle$ is a model, and it satisfies conditions (a)--(d).
\end{proof}

The construction presented in Lemma~\ref{l:filtrationalg} can be generalized as follows. Let $\langle \mathbf{A}, \mu \rangle$ be an N-algebra with a valuation and ${\rm \Sigma}$ be a finite set of formulas closed under subformulas. Suppose that $L$ is the universe of a finite sublattice of $\mathbf{A}$ with unit $1_{\mathbf{A}}$ such that $\mu[{\rm \Sigma}] \subseteq L$. Consider the N-algebra $\mathbf{L} = \langle L, \land, \lor, \to_{\mathbf{L}}, \neg_{\mathbf{L}}, 1_{\mathbf{A}} \rangle$, where $\to_{\mathbf{L}}$ and $\neg_{\mathbf{L}}$ are defined as in~(\ref{eq:impl}) and~(\ref{eq:neg}). Let $\mu_{\mathbf{L}}$ be a valuation on $\mathbf{L}$ such that $\mu_{\mathbf{L}}(p) = \mu(p)$ for each variable $p \in {\rm \Sigma}$. We call the pair $(\mathbf{L}, \mu_{\mathbf{L}})$ a \emph{filtration} of $(\mathbf{A}, \mu)$ through ${\rm \Sigma}$.

\begin{theorem}[Filtration Theorem]
If $(\mathbf{L}, \mu_{\mathbf{L}})$ is a filtration of $(\mathbf{A}, \mu)$ through ${\rm \Sigma}$, then $\mu_{\mathbf{L}}(\varphi) = \mu(\varphi)$ for each $\varphi \in {\rm \Sigma}$.
\end{theorem}

\begin{proof}
The proof goes again by induction, and we focus on the case $\neg \varphi \in {\rm \Sigma}$. But then, $\neg \mu (\varphi) = \mu(\neg \varphi) \in \mu[{\rm \Sigma}] \subseteq L$. Therefore, $\neg \mu(\varphi) = \neg_{\mathbf{L}} \mu (\varphi)$. Thus, $\mu_{\mathbf{L}}(\neg \varphi) = \neg \mu_{\mathbf{L}}(\varphi) = \neg \mu(\varphi) = \mu(\neg \varphi)$.  
\end{proof}

\noindent Among the filtrations $(\mathbf{L}, \mu_{\mathbf{L}})$ of $(\mathbf{A}, \mu)$, the filtration $(\mathbf{S}, \mu_{\mathbf{S}})$ is the least one. The following shows that $(\mathbf{S}, \mu_{\mathbf{S}})$ corresponds to the greatest filtration $\mathfrak{M}_{\mathbf{A}}^{g}$ of $\mathfrak{M}_{\mathbf{A}}$ through ${\rm \Sigma}$. Consider $[w], [v] \in W_{\mathbf{A}}^{*}$ such that $[w] \leq^{g} [v]$. But then, by~(\ref{eq:biggest}), this means that $(\forall \varphi \in {\rm \Sigma})(w \in V(\varphi) \Rightarrow v \in V(\varphi))$. This is again equivalent to $w \cap S \subseteq v \cap S$. Hence, we conclude $[w] \leq^{*} [v]$. 


\section{Modal Companions}\label{s:modal}

A normal extension $\mathsf{M}$ of $\mathsf{S4}$ is a \emph{modal companion} of an intermediate logic {\sf L} if for each propositional formula $\varphi$ we have $\mathsf{L}\vdash\varphi$ iff $\mathsf{M}\vdash \varphi^\Box$, where 
$\varphi^\Box$ is the G\"odel translation of $\varphi$. In this section we define (bi-)modal companions of subminimal logics. The negation operator in {\sf N} and its extensions behave like a non-normal modal operator which suggests that bi-modal companions of subminimal logics must have a non-normal modal operator. For the theory of non-normal modal logics and neighbourhood semantics we refer the reader to~\cite{chellas1980modal,hansen2003monotonic,pacuit2017neighborhood}.

Consider the bi-modal language $\mathcal{L}_{\square}({\sf Prop})$, where ${\sf Prop}$ is a countable set of propositional variables, generated by the following grammar: 
\[
p \mid \bot\mid \top \mid \varphi \land \varphi \mid \varphi \lor \varphi \mid \varphi \to \varphi \mid \square\varphi \mid \blacksquare \varphi
\]  
\noindent 
where $p \in {\sf Prop}$. We write $\neg \varphi$ to denote the implication $\varphi \to \bot$.  
Recall that the axioms and rules for the modal logic ${\sf S4}$ are:
\[
\begin{array}{rl}
{\rm (K)} & \square (p \to q) \to (\square p \to \square q);\\[.1in]
{\rm (T)} & \square p \to p;\\[.1in]
{\rm (4)} & \square p \to \square\square p,\\[.1in]
\end{array}
\]

\noindent in addition to all the classical tautologies, and the rules of \emph{modus ponens}, uniform substitution, and necessitation ($p / {\square p}$). 
Now, consider the following additional bi-modal axioms: 
\begin{equation}\label{eq:modaln}
\square (p \leftrightarrow q) \to (\blacksquare p \leftrightarrow \blacksquare q),
\end{equation}
\begin{equation}\label{eq:modal2}
\blacksquare p \to \square \blacksquare p.
\end{equation}

\noindent Note that the $\blacksquare$ modality is not a normal modality, since we do not have the rule of necessitation for it. As for the notation, we denote the new modality as a \emph{box} modality since it behaves as a universal modality from the point of view of neighbourhood semantics. We denote this bi-modal system by ${\sf NS4}$. 

A neighbourhood frame for $\emph{\sf NS4}$ is a triple $\mathfrak{F} = \langle W, \leq , \mathfrak{n} \rangle$, where $\langle W, \leq \rangle$ is a set equipped with a reflexive transitive relation $\leq$, and $\mathfrak{n}$ a monotone function $\mathfrak{n} \colon W \to \mathcal{P}(\mathcal{P}(W))$ (i.e.,\ if $w\leq v$ then $\mathfrak{n}(w)\subseteq \mathfrak{n}(v)$) such that: 
\begin{equation}\label{eq:localm}
X \in \mathfrak{n}(w) \Longleftrightarrow X \cap R(w) \in \mathfrak{n}(w).
\end{equation}
Using the fact that a neighbourhood function $\mathfrak{n}$ induces the existence of a function ${\rm N} \colon \mathcal{P}(W) \to \mathcal{U}(W,\leq)$ via the equivalence $X \in \mathfrak{n}(w)$ if and only if $w\in {\rm N}(X)$, we consider the following generalization of N-semantics.

An \emph{N-frame for} \textsf{NS4} is a triple $\mathfrak{F} = \langle W, \leq , {\rm N} \rangle$, where $\langle W, \leq \rangle$ is a set equipped with a reflexive transitive relation, and ${\rm N} \colon \mathcal{P}(W) \to \,\mathcal{U}(W,\leq)$ is a map such that:
\begin{equation}\label{eq:localm}
w \in {\rm N}(X) \Longleftrightarrow w\in {\rm N}(X \cap R(w) ).
\end{equation}
\noindent
An \emph{N-model for} {\sf NS4} is a pair $\mathfrak{M} = \langle \mathfrak{F}, V \rangle$, where $\mathfrak{F}$ is an N-frame and $V$ is a valuation $V\colon {\sf Prop} \to \mathcal{P}(W)$. Truth of a formula is defined as follows:

\[
\begin{array}{lcl}
\mathfrak{M}, w \models p & \Longleftrightarrow & w \in V(p);\\[.1in]
\mathfrak{M}, w \models \varphi \land \psi & \Longleftrightarrow & \mathfrak{M}, w \models \varphi \text{ and }\mathfrak{M}, w \models \psi;\\[.1in]
\mathfrak{M}, w \models \varphi \lor \psi & \Longleftrightarrow & \mathfrak{M}, w \models \varphi \text{ or }\mathfrak{M}, w \models \psi;\\[.1in]
\mathfrak{M}, w \models \varphi \to \psi  & \Longleftrightarrow & \mathfrak{M}, w \models \varphi \text{ implies }\mathfrak{M}, w \models \psi;\\[.1in]
\mathfrak{M}, w \models \square \varphi & \Longleftrightarrow & R(w) \subseteq V(\varphi);\\[.1in]
\mathfrak{M}, w \models \blacksquare \varphi & \Longleftrightarrow & w \in {\rm N}(V(\varphi)),\\[.1in]
\end{array}
\]
where $V(\varphi) \coloneqq \{w \in W \colon \mathfrak{M}, w \models \varphi\}$. A formula $\varphi$ is said to be valid in a model $\mathfrak{M}$ if $\mathfrak{M}, w \models \varphi$ for every $w \in W$, and $\varphi$ is valid in a frame $\mathfrak{F}$ if every model on $\mathfrak{F}$ validates $\varphi$. We let $\mathcal{F}$ denote the class of all N-frames for {\sf NS4}. We will now show that  ${\sf NS4}$ is sound with respect to $\mathcal{F}$.

\begin{theorem}[Soundness]
The system ${\sf NS4}$ is sound with respect to $\mathcal{F}$.
\end{theorem}

\begin{proof}
We only check the cases for the axioms~(\ref{eq:modaln}) and~(\ref{eq:modal2}), by proving that they are valid on each frame. 
The validity of~(\ref{eq:modal2}) follows immediately from the fact that ${\rm N}(X)$ is upward closed for every $X \in \mathcal{P}(W)$. For~(\ref{eq:modaln}), suppose that $\langle \mathfrak{F}, V \rangle, w \models \square(p \leftrightarrow q)$. This means that $R(w) \subseteq V(p \leftrightarrow q)$, which is in turn equivalent to $V(p) \cap R(w) = V(q) \cap R(w)$. In order to conclude $\langle \mathfrak{F}, V \rangle, w \models \blacksquare p \leftrightarrow \blacksquare q$, we assume without loss of generality that $w \in {\rm N}(V(p))$. By~(\ref{eq:localm}), this means that $w \in {\rm N}(V(p)\cap R(w))$. But then, $w \in {\rm N}(V(q) \cap R(w))$, since $V(p) \cap R(w) = V(q) \cap R(w)$. Thus, again by~(\ref{eq:localm}), we conclude $w \in {\rm N}(V(q))$. 
\end{proof}

We now focus on showing that {\sf NS4} is complete with respect to $\mathcal{F}$. To do this, we follow the standard approach for normal modal logics (see, e.g.,~\cite{CZ97,blackburn2002modal}), and combine it with the standard approach for neighbourhood semantics~\cite{pacuit2017neighborhood}. Recall that a set ${\rm \Gamma}$ of formulas is said to be maximally consistent if it is consistent (in the sense that it does not contain both $\varphi$ and $\neg \varphi$, for any $\varphi$) and, for every formula $\varphi$, either $\varphi \in {\rm \Gamma}$ or $\neg \varphi \in {\rm \Gamma}$. Every consistent set of formulas can be extended to a maximally consistent set of formulas (Lindenbaum's Lemma). As a consequence, it can be proved that every formula that is contained in all maximally consistent sets of formulas is a theorem. Consider now the set $\mathcal{W}$ of all maximally consistent sets of {\sf NS4} formulas. We denote by $| \varphi |$ the set of all maximally consistent sets of formulas containing $\varphi$. We define the canonical relation $\leq $ by: ${\rm \Gamma} \leq {\rm \Delta}$ if, whenever $\square \varphi \in {\rm \Gamma}$, then $\varphi \in {\rm \Delta}$. A map $\mathcal{N} \colon \mathcal{P}(\mathcal{W}) \to \mathcal{U}(\mathcal{W},\leq)$ is a canonical N-function provided that for all $\varphi \in \mathcal{L}_{\square}({\sf Prop})$, 
\[
 {\rm \Gamma} \in \mathcal{N}(|\varphi|) \Longleftrightarrow \blacksquare \varphi \in {\rm \Gamma}.
\]
We consider the canonical valuation $\mathcal{V}\colon {\rm Prop} \to \mathcal{P}(\mathcal{W})$ defined by 
\[
\mathcal{V}(p) = |p| \coloneqq \{{\rm \Gamma} \in \mathcal{W} \colon p \in {\rm \Gamma}\}.
\] 

\noindent The quadruple $\langle \mathcal{W}, \leq , \mathcal{N}, \mathcal{V} \rangle$ is called a canonical model. To prove completeness of ${\sf NS4}$, we use the following function
\[
\mathcal{N}^{loc}(X) \coloneqq \{{\rm \Gamma} \in \mathcal{W} \, | \, (\exists \blacksquare \psi \in {\rm \Gamma})(X \cap R({\rm \Gamma}) = | \psi | \cap R({\rm \Gamma}))\}.
\]
which makes $\mathcal{M} = \langle \mathcal{W},\leq , \mathcal{N}^{loc}, \mathcal{V} \rangle$ into a canonical model, and its underlying frame into an {\sf NS4}-frame. First of all, it is easy to see that $\mathcal{N}^{loc}$ satisfies~(\ref{eq:localm}). Further, the set $\mathcal{N}^{loc}(X)$ is upward closed for every $X \in \mathcal{P}(\mathcal{W})$. In fact, if $\blacksquare \psi \in {\rm \Gamma}$, then also $\square\blacksquare \psi \in {\rm \Gamma}$.  
But then, ${\rm \Gamma}\leq {\rm \Delta}$ implies $\blacksquare \psi \in {\rm \Delta}$. Since $X \cap R({\rm \Delta}) = |\psi| \cap R({\rm \Delta})$, we conclude. Moreover, $\mathcal{M}$ is canonical. In fact, ${\rm \Gamma}\in \mathcal{N}^{loc}(|\varphi|)$ implies that $|\psi| \cap R({\rm \Gamma}) = | \varphi | \cap R({\rm \Gamma})$ for some $\blacksquare \psi \in {\rm \Gamma}$, which means that for every successor ${\rm \Delta}$ of ${\rm \Gamma}$, $\varphi \leftrightarrow \psi \in {\rm \Delta}$ and, therefore, $\square (\varphi \leftrightarrow \psi) \in {\rm \Gamma}$. By the axiom~(\ref{eq:modaln}) and the fact that $\blacksquare \psi \in {\rm \Gamma}$, we conclude $\blacksquare \varphi \in {\rm \Gamma}$. A canonical function $\mathcal{N}$ is what ensures that membership and truth coincide in the canonical model (Truth Lemma); the cases for the other formulas proceed by the standard induction. 

\begin{theorem}[Completeness]
The system ${\sf NS4}$ is complete with respect to $\mathcal{F}$.
\end{theorem}

The next step is to show that {\sf NS4} is the modal companion of {\sf N}. We start by recalling that it is possible to translate formulas from the intuitionistic language into the modal language via the G\"odel translation: 
\[
\begin{array}{rcl}
\bot^{\square} & = & \bot;\\[.1in]
p^{\square} & = & \square p;\\[.1in]
(\varphi \circ \psi)^{\square} & = & \varphi^{\square} \circ \psi^{\square}\text{, where }\circ \in \{\land, \lor\};\\[.1in]
(\varphi \to \psi)^{\square}  & = & \square (\varphi^{\square} \to \psi^{\square}). \\[.1in]
\end{array}
\]
The celebrated G\"odel-McKinsey-Tarski theorem states that for every intuitionistic formula $\varphi$, 
\[
{\sf IPC} \vdash \varphi \Longleftrightarrow {\sf S4} \vdash \varphi^{\square}. 
\]
\noindent The logic {\sf S4} is not the only modal companion of {\sf IPC} (e.g., Grzegorczyk's logic {\sf Grz} is another well-known example). A fundamental fact to be used in the proof is that any intuitionistic Kripke model $\mathfrak{M}$ can be seen as a model for ${\sf S4}$, and for every node in the model the following holds: 
\[
\mathfrak{M}, w \models \varphi \Longleftrightarrow \mathfrak{M}, w \models_{\square} \varphi^{\square},
\]
\noindent 
where $\models$ denotes the intuitionistic and $\models_{\square}$ the modal notion of truth. Similarly, every  {\sf N}-model has an equivalent {\sf NS4}-model, as any {\sf N}-frame $\langle W, \leq, {\rm N}\rangle$ induces a corresponding {\sf NS4}-frame $\langle W, \leq, {\rm N}^*\rangle$, with ${\rm N}^*\colon \mathcal{P}(W) \to \mathcal{U}(W, \leq)$ defined by:
\[
{\rm N}^*(X) \coloneqq \{w \in W \mid (\exists Y \in \mathcal{U}(W, \leq))(X \cap R(w) = Y\cap R(w)\text{ and }w \in {\rm N}(Y))\}.
\]
We start by adapting the G\"odel translation in order to obtain a translation of {\sf N} into {\sf NS4}. It is sufficient to add a translation clause for the negation in the language of {\sf N}:
\[
(\neg\varphi)^{\square} = \blacksquare \varphi^{\square},
\] 
and to discard the clause for $\bot$. The following preliminary result is easy to prove.

\begin{lemma}\label{l:frame}
Let $\mathfrak{M} = \langle W, \leq , {\rm N}, V \rangle$ be a model for {\sf N}, and let $\mathfrak{M}^* = \langle W, \leq , {\rm N}^*, V \rangle$ be the corresponding N-model for {\sf NS4} defined in the way explained above. Then, for every formula $\varphi$,
\[
\mathfrak{M}, w \models \varphi \Longleftrightarrow \mathfrak{M}^*, w \models_{\square} \varphi^{\square},
\]
where $\models$ and $\models_{\square}$ are, respectively, the {\sf N} and the modal notion of truth.
\end{lemma}

\noindent In what follows, we allow ourselves to identify $\square \varphi \land \square \psi$ and $\square(\varphi \land \psi)$ for formulas $\varphi$ and $\psi$, omitting thereby some trivial steps in the derivations.

\begin{theorem}
For every formula $\varphi \in \mathcal{L}({\sf Prop})$,
\[
\mathsf{N} \vdash \varphi \Longleftrightarrow {\sf NS4} \vdash \varphi^{\square}.
\]
\end{theorem}

\begin{proof}


From left to right, we show that ${\sf NS4} \vdash \big{(} (p \leftrightarrow q) \to (\neg p \leftrightarrow \neg q) \big{)}^{\square}$, that is, ${\sf NS4} \vdash \square \big{(}\square(\square p \leftrightarrow \square q) \to \square (\blacksquare \square p \leftrightarrow \blacksquare \square q)\big{)}.$ Indeed:
\[
\begin{array}{lclr}
& \vdash & \square (\square p \leftrightarrow \square q) \to (\blacksquare \square p \leftrightarrow \blacksquare \square q) & \text{By axiom }(\ref{eq:modaln})\\[.1in]
& \vdash & \square\square (\square p \leftrightarrow \square q) \to \square (\blacksquare \square p \leftrightarrow \blacksquare \square q) \,\,\,\,\,\,\,&\,\,\,\,\,\,\,\, \text{By necessitation}\\ [.1in]
& \vdash &  \square (\square p \leftrightarrow \square q) \to  \square\square (\square p \leftrightarrow \square q) & \text{By (4)}\\ [.1in]
& \vdash &  \square (\square p \leftrightarrow \square q) \to \square (\blacksquare \square p \leftrightarrow \blacksquare \square q)&\\ [.1in]
& \vdash &  \square\big{(}\square (\square p \leftrightarrow \square q) \to \square (\blacksquare \square p \leftrightarrow \blacksquare \square q)\big{)}&\,\,\,\,\,\,\,\, \text{By necessitation}
\end{array}
\]
For the other direction, suppose that there exists an {\sf N}-countermodel of a formula $\varphi$. By Lemma~\ref{l:frame}, this leads to a countermodel for $\varphi^{\square}$ in {\sf NS4}, as desired.  
\end{proof}

The next part of the section aims for a better understanding of the bi-modal logic {\sf NS4} by studying the behaviour of the modality $\blacksquare$ resulting from the subminimal negation. We hence focus now on proving the following main result. 
\\
\\
\textbf{Theorem~\ref{t:blackbox}.}
{\em The $\{\land, \lor, \to, \blacksquare\}$-fragment of {\sf NS4} is axiomatized by the following countably many rules:
\[
\infer[{\rm R}_n]{(\blacksquare p_1 \land \dots \land \blacksquare p_n) \to (\blacksquare q \leftrightarrow \blacksquare r)}{(\blacksquare p_1 \land \dots \land \blacksquare p_n) \to (q \leftrightarrow r)}
\]
for each $n \in \mathbb{N}$.}
\\
\\
\noindent In what follows, we call the logic axiomatized by this countable set of rules ${\sf E_\mathbb{N}}$. This notation is justified by the fact that \emph{classical modal logic} is the non-normal modal logic axiomatized by:
\[
\infer[]{\square p \leftrightarrow \square q}{p \leftrightarrow q}
\]
and is denoted by {\sf E}. From the perspective of the N-semantics, the defining rule of {\sf E} is expressed by the fact that {\rm N} lifts to a well-defined function on sets.

First, we give a proof that each of these rules is derivable in {\sf NS4}. For each rule, we provide a derivation of the conclusion from the premise in {\sf NS4}.

\begin{lemma}
The rule 
\[
\infer[{\rm R}_n]{(\blacksquare p_1 \land \dots \land \blacksquare p_n) \to (\blacksquare q \leftrightarrow \blacksquare r)}{(\blacksquare p_1 \land \dots \land \blacksquare p_n) \to (q \leftrightarrow r)}
\]
is derivable in {\sf NS4}, for each $n \in \mathbb{N}$.
\end{lemma}

\begin{proof} 
The case $n=0$ is clear. For $n>0$, assume $(\blacksquare p_1 \land \dots \land \blacksquare p_n) \to (q \leftrightarrow r)$. But then:
\[
\begin{array}{lclr}
& \vdash & \blacksquare p_i \to \square\blacksquare p_i, \text{ for }i \in \{1,\ldots,n\} & \text{By axiom }(\ref{eq:modal2}) \\[.1in]
& \vdash & (\blacksquare p_1 \land \dots \land \blacksquare p_n) \to (\square\blacksquare p_1 \land \dots \land \square \blacksquare p_n) &\\ [.1in]
& \vdash & (\blacksquare p_1 \land \dots \land \blacksquare p_n) \to \square(\blacksquare p_1 \land \dots \land \blacksquare p_n) & {\rm (a)}\\ [.1in]
& \vdash & \square (\blacksquare p_1 \land \dots \land \blacksquare p_n) \to \square (q \leftrightarrow r) & \,\,\,\,\,\,\,\,\,\text{By necessitation}\\[.1in]
& \vdash & \square (q \leftrightarrow r) \to (\blacksquare q \leftrightarrow \blacksquare r) & \text{By axiom }(\ref{eq:modaln}) \\[.1in]
& \vdash & \square (\blacksquare p_1 \land \dots \land \blacksquare p_n) \to (\blacksquare q \leftrightarrow \blacksquare r) & \text{(b)}\\[.1in]
& \vdash & (\blacksquare p_1 \land \dots \land \blacksquare p_n) \to (\blacksquare q \leftrightarrow \blacksquare r) & \text{By {\rm (a)} and {\rm (b)}}
\end{array}
\]
\end{proof}

\noindent As a consequence, we obtain the following partial result:

\begin{lemma}\label{l:sounde}
Given a $\{\land,\lor,\to,\blacksquare\}$-formula $\varphi$, 
\[
{\sf E}_{\mathbb{N}} \vdash \varphi \Longrightarrow {\sf NS4} \vdash \varphi.
\]
\end{lemma}

In order to conclude the proof of Theorem~\ref{t:blackbox} we take an ${\sf E}_{\mathbb{N}}$\nbd{-}countermodel for a $\{\land,\lor,\to,\blacksquare\}$-formula $\varphi$, and build from it an {\sf NS4}-countermodel for $\varphi$. Therefore, we first provide a completeness result for the logic ${\sf E}_{\mathbb{N}}$. 

Consider the modal N-frames $\langle W, {\rm N} \rangle$ where, for every $X, Y, Z_1, \ldots, Z_n \in \mathcal{P}(W)$ and $n \in \mathbb{N}$, the function ${\rm N} \colon \mathcal{P}(W) \to \mathcal{P}(W)$ satisfies $({\rm E}_n)$:
\[
\mathrm{N}(X)\cap \mathrm{N}(Z_1)\cap\dots\cap \mathrm{N}(Z_n)\,{=}\,\mathrm{N}(X\cap \mathrm{N}(Z_1)\cap\dots\cap \mathrm{N}(Z_n))\cap \mathrm{N}(Z_1)\cap\dots\cap \mathrm{N}(Z_n).
\]
The corresponding notion of model is defined in the standard way. 
We next show that ${\rm R}_n$ is characterized by ${\rm E}_n$.
\begin{proposition}
For any $n \in \mathbb{N}$, a modal N-frame $\mathfrak{F} = \langle W, {\rm N} \rangle$ satisfies ${\rm E}_n$ if and only if $\mathfrak{F} \models {\rm R}_n$.
\end{proposition}

\begin{proof}
From left to right, consider a frame $\mathfrak{F} = \langle W, {\rm N} \rangle$ satisfying ${\rm E}_n$, and assume that for a given valuation $V$, 
\[
\langle \mathfrak{F}, V \rangle \models (\blacksquare p_1 \land \dots \land \blacksquare p_n) \to (q \leftrightarrow r).
\]
This means that:
\[
V(q)\cap \mathrm{N}(V(\blacksquare p_1))\cap\dots\cap \mathrm{N}(V(\blacksquare p_n))=V(r)\cap \mathrm{N}(V(\blacksquare p_1))\cap\dots\cap \mathrm{N}(V(\blacksquare p_n)).
\]
Now, suppose $\langle \mathfrak{F}, V \rangle, w \models \blacksquare p_1 \land \dots \land \blacksquare p_n$, that is,
\[
w \in {\rm N}(V(p_1)) \cap \dots \cap {\rm N}(V(p_n)).
\]
It is now sufficient to prove that $w\models\blacksquare q$ if and only if $w\models\blacksquare r$. Indeed, without loss of generality, $w \in \mathrm{N}(V(q))$ entails 
\[
w\in \mathrm{N}(V(q))\cap \mathrm{N}(V(p_1))\cap\dots\cap \mathrm{N}(V(p_n)).
\]
Further, by ${\rm E}_n$ we get 
\[
w\in \mathrm{N}(V(q)\cap \mathrm{N}(V( p_1))\cap\dots\cap \mathrm{N}(V( p_n)))\cap \mathrm{N}(V( p_1))\cap\dots\cap \mathrm{N}(V( p_n))
\]
that is equivalent to 
\[
w\in \mathrm{N}(V(r)\cap \mathrm{N}(V( p_1))\cap\dots\cap \mathrm{N}(V( p_n)))\cap \mathrm{N}(V( p_1))\cap\dots\cap \mathrm{N}(V( p_n))
\]
by our assumption. Therefore, by ${\rm E}_n$ again, $w\in \mathrm{N}(V(r))$. It follows that $\mathfrak{F} \models {\rm R}_n$.

For the right-to-left direction, we assume $\mathfrak{F} \models {\rm R}_n$, and suppose that there exist subsets $X, Y, Z_1, \ldots, Z_n$ such that
\begin{equation}\label{eq:contradiction}
\mathrm{N}(X)\cap \mathrm{N}(Z_1)\cap\dots\cap \mathrm{N}(Z_n)\,{\ne}\,\mathrm{N}(X\cap \mathrm{N}(Z_1)\cap\dots\cap \mathrm{N}(Z_n))\cap \mathrm{N}(Z_1)\cap\dots\cap \mathrm{N}(Z_n).
\end{equation}
Set $V(p_i)=Z_i$ for $i \in \{1, \ldots, n\}$, $V(q)=X$, and $V(r)=X\cap {\rm N}(Z_1)\cap\dots\cap {\rm N}(Z_n)$, for propositional variables $p_i, q, r$. But then, $V(\blacksquare p_i)={\rm N}(Z_i)$, $V(\blacksquare q)={\rm N}(X)$ and 
\[
V(\blacksquare r)={\rm N}(X\cap {\rm N}(Z_1)\cap\dots\cap {\rm N}(Z_n)).
\]
Take any $w\in {\rm N}(Z_1)\cap\dots\cap {\rm N}(Z_n)$. Obviously then $w\in X$ is equivalent to $w\in X\cap {\rm N}(Z_1)\cap\dots\cap {\rm N}(Z_n)$, and hence, $(\blacksquare p_1\land\dots\land\blacksquare p_n)\to (q\leftrightarrow r)$ holds in $\mathfrak F$. But then,  $(\blacksquare p_1\land\dots\land\blacksquare p_n)\to (\blacksquare q\leftrightarrow\blacksquare  r)$ holds in $\mathfrak F$ as well by ${\rm R}_n$, which contradicts~(\ref{eq:contradiction}).  
\end{proof}

\noindent We conclude the completeness proof by means of a canonical model construction. Consider again the set $\mathcal{W}$ of maximally consistent sets of $\{\land,\lor,\to,\blacksquare\}$-formulas. For any $\rm\Gamma\in\mathcal{W}$, set $\mathcal{R}(\rm\Gamma)\coloneqq\{\rm{\Delta} \in \mathcal{W}\,|\,\forall\chi(\blacksquare\chi\in\rm{\Gamma}\Rightarrow \blacksquare\chi\in\rm{\Delta})\}$ and define 
\begin{equation}\label{eq:Ndef}
\mathcal{N}(X) \coloneqq \{{\rm \Gamma}\in\mathcal{W} \, | \, (\exists \blacksquare \psi \in {\rm \Gamma})(X \cap \mathcal{R}(\rm\Gamma) = | \psi | \cap \mathcal{R}(\rm\Gamma))\},
\end{equation}
for any $X\in\mathcal{P(W)}$. We need to prove that $\langle \mathcal{W}, \mathcal{N} \rangle$ is an N-frame satisfying ${\rm E}_n$ for each $n \in \mathbb{N}$. Consider a sequence of subsets $X, Z_1,\dots,Z_n \subseteq \mathcal{W}$, and take ${\rm \Gamma} \in \mathcal{N}(Z_1)\cap \ldots\cap \mathcal{N}(Z_n)$. To show that E$_n$ holds it will be sufficient to prove 
\begin{equation}\label{eq:Ndefmember}
{\rm\Gamma}\in\mathcal{N}(X)\mbox{ if and only if  }{\rm\Gamma}\in\mathcal{N}(X\cap \mathcal{N}(Z_1)\cap\dots\cap \mathcal{N}(Z_n)).
\end{equation}
By~(\ref{eq:Ndef}), ${\rm \Gamma} \in \mathcal{N}(Z_1)\cap \ldots\cap \mathcal{N}(Z_n)$ means that, for each $i \in \{1, \ldots, n\}$, \mbox{there exists }$\blacksquare\psi_i\in{\rm\Gamma}$\mbox{ such that }
\begin{equation}\label{black}
Z_i\cap\mathcal{R}({\rm\Gamma})=| \psi_i | \cap \mathcal{R}(\rm\Gamma).
\end{equation}
To prove~(\ref{eq:Ndefmember}), it suffices to show that ${\rm\Delta}\in X\cap \mathcal{R}(\rm\Gamma)$ implies ${\rm\Delta}\in \mathcal{N}(Z_i)$ 
for every ${\rm\Delta} \in \mathcal{W}$ and every $i \in \{1, \ldots, n\}$.
So, assume ${\rm\Delta}\in X\cap \mathcal{R}(\rm\Gamma)$. But  ${\rm\Delta}\in \mathcal{R}(\rm\Gamma)$ means, by~(\ref{black}), that for each $i \in \{1, \ldots, n\}$, $\blacksquare\psi_i\in{\rm\Delta}$. Moreover, $\mathcal{R}(\rm\Delta)\subseteq\mathcal{R}(\rm\Gamma)$. So
\[
Z_i\cap\mathcal{R}({\rm\Delta})=Z_i\cap\mathcal{R}({\rm\Gamma})\cap\mathcal{R}({\rm\Delta})=| \psi_i | \cap \mathcal{R}(\rm\Gamma)\cap\mathcal{R}({\rm\Delta})=|\psi_i|\cap\mathcal{R}({\rm\Delta}),
\] 
and hence, ${\rm\Delta}\in \mathcal{N}(Z_i)$ for each $i \in \{1, \ldots, n\}$ by~(\ref{black}) again. By defining now the canonical valuation as $\mathcal{V}(p) = |p|$, it is standard to prove that $\varphi \in {\rm \Gamma}$ if and only if $\mathcal{M}, {\rm \Gamma} \models \varphi$, with $\mathcal{M} = \langle \mathcal{W}, \mathcal{N}, \mathcal{V} \rangle$. Therefore:

\begin{proposition}\label{p:complblack}
The logic ${\sf E}_{\mathbb{N}}$ is sound and complete with respect to the {\rm N}-models satisfying $\{{\rm E}_n \colon n \in \mathbb{N}\}$.  
\end{proposition}

Now, we endow the canonical model of ${\sf E}_{\mathbb{N}}$ with a reflexive transitive relation so as to obtain an {\sf NS4}-model. Set ${\rm \Gamma} \leq {\rm \Delta}$ if and only if ${\rm \Delta}\in \mathcal{R}({\rm \Gamma})$. It is clear that $\mathcal{N}(X) \in \mathcal{U}(\mathcal{W},\leq)$, for any $X \in \mathcal{P}(\mathcal{W})$. It suffices to show that ${\rm \Gamma}\in\mathcal{N}({X})$ is equivalent to ${\rm \Gamma}\in\mathcal{N}({X}\cap\mathcal{R}({\rm \Gamma}))$, for ${\rm \Gamma} \in \mathcal{W}$. Indeed, we have ${\rm \Gamma}\in\mathcal{N}(X\cap\mathcal{R}({\rm \Gamma}))$ if and only if $X \cap \mathcal{R}(\rm\Gamma) \cap \mathcal{R}(\rm\Gamma) = | \psi | \cap \mathcal{R}(\rm\Gamma)$ for some $\blacksquare \psi \in {\rm \Gamma}$, that is, $X \cap \mathcal{R}(\rm\Gamma) = | \psi | \cap \mathcal{R}(\rm\Gamma)$. Therefore, ${\rm \Gamma}\in\mathcal{N}({X})$. The following is now immediate:

\begin{lemma}\label{l:compl}
Given a $\{\land,\lor,\to,\blacksquare\}$-formula $\varphi$, 
\[
{\sf NS4} \vdash \varphi \Longrightarrow {\sf E}_{\mathbb{N}} \vdash \varphi.
\]
\end{lemma}

At this point, Theorem~\ref{t:blackbox} follows from Lemmas~\ref{l:sounde} and~\ref{l:compl}.

\begin{theorem}\label{t:blackbox}
The $\{\land, \lor, \to, \blacksquare\}$-fragment of {\sf NS4} is axiomatized by the following countably many rules:
\[
\infer[{\rm R}_n]{(\blacksquare p_1 \land \dots \land \blacksquare p_n) \to (\blacksquare q \leftrightarrow \blacksquare r)}{(\blacksquare p_1 \land \dots \land \blacksquare p_n) \to (q \leftrightarrow r)}
\]
for each $n \in \mathbb{N}$.
\end{theorem}


We conclude the section, and the whole article, with a brief sketch of how to obtain a modal companion for contraposition logic {\sf CoPC}. We recall that it is defined by the axiom $(p \to q) \to (\neg q \to \neg p)$. We consider the bi-modal logic {\sf CoS4} obtained by replacing the axiom~(\ref{eq:modaln}) by  
\begin{equation}\label{eq:modalcopc}
\square (p \to q ) \to (\blacksquare q \to \blacksquare p).
\end{equation}
The corresponding N-semantics is given by frames $\mathfrak{F} =  \langle W, \leq , {\rm N} \rangle$, where $\langle W, \leq \rangle$ is a poset, and {\rm N} is antitone, i.e., for $X, Y \in \mathcal{P}(W)$,
\[
X \subseteq Y \Longrightarrow {\rm N}(Y) \subseteq{\rm N}(X).
\]

\noindent The logic {\sf CoPC} can be translated into the system obtained by replacing~(\ref{eq:modaln}) by~(\ref{eq:modalcopc}) via the translation defined at the beginning of the section.

\begin{theorem}
For every formula $\varphi$ we have 
\[
\mathsf{N} \vdash \varphi \Longleftrightarrow {\sf CoS4} \vdash (\varphi)^{\square}.
\]
\end{theorem}

\begin{proof}
For the left-to-right direction, consider the translation of the axiom $(p \to q) \to (\neg q \to \neg p)$, i.e., $\square\big{(}\square (\square p \to \square q) \to \square (\blacksquare \square q \to \blacksquare \square p)\big{)}.$ We have: 
\[
\begin{array}{lclr}
& \vdash & \square (\square p \to \square q) \to (\blacksquare \square q \to \blacksquare \square p) & \text{By axiom }(\ref{eq:modalcopc})\\[.1in]
& \vdash & \square\square (\square p \to \square q) \to \square (\blacksquare \square q \to \blacksquare \square p) \,\,\,\,\,\,\,&\,\,\,\,\,\,\,\, \text{By necessitation}\\ [.1in]
& \vdash &  \square (\square p \to \square q) \to  \square\square (\square q \to \square p) & \text{By (4)}\\ [.1in]
& \vdash &  \square (\square p \to \square q) \to \square (\blacksquare \square q \to \blacksquare \square p)&\\ [.1in]
& \vdash &  \square\big{(}\square (\square p \to \square q) \to \square (\blacksquare \square q \to \blacksquare \square p)\big{)}&\,\,\,\,\,\,\,\, \text{By necessitation}
\end{array}
\]
For the reverse direction, it suffices again to observe that from any N-frame for {\sf CoPC} we can obtain an N-frame for {\sf CoS4}.  
\end{proof}

We finish by highlighting a few further questions and research directions. First,  it is reasonable to expect the $\{\land, \lor, \to, \blacksquare\}$-fragment of {\sf CoS4} to be axiomatized by the following countably many rules
\[
\infer[]{(\blacksquare p_1 \land \dots \land \blacksquare p_n) \to (\blacksquare r \to \blacksquare q)}{(\blacksquare p_1 \land \dots \land \blacksquare p_n) \to (q \to r)}
\]
for $n \,{\in}\, \mathbb{N}$, the argument being similar to the one for Theorem~\ref{t:blackbox}. This raises the question whether there exist finite axiomatizations for the $\{\land, \lor, \to, \blacksquare\}$\nbd{-}fragments of the considered bi-modal systems.

More generally, the notion and properties of bi-modal companions of  subminimal logics naturally lead to the question whether the theory of modal companions of intermediate logics (\cite[Section 9.6]{CZ97}) can be paralleled in this case. 
For example, do there always exist the least and greatest bi-modal companions of subminimal logics? Can one prove an analogue of the Blok-Esakia theorem in this setting? Also the study of the interplay between the two modalities could be taken further through a comparison with widely studied non-normal modal logics, such as classical modal logic {\sf E} and monotonic modal logic {\sf M}.

Finally, it will be interesting to investigate the relations between our semantics for the logic {\sf CoPC} and the one proposed by Hazen~\cite{Haz95}. 
\vspace*{1cm} \\
\newpage

\end{document}